\documentclass[reqno, 12pt]{amsart}
\usepackage{enumerate}
\usepackage{amssymb,url,color}
\usepackage{hyperref}
\usepackage{amsmath,amsthm}

\allowdisplaybreaks[4]

\newtheorem{thm}{Theorem}[section]
\newtheorem{cor}{Corollary}[section]

\newtheorem{lem}{Lemma}[section]
\newtheorem{rem}{Remark}[section]
\newtheorem{exa}{Example}[section]

\theoremstyle{definition}

\numberwithin{equation}{section}
\newcommand{\pp}{\mathbb{P}}
\newcommand{\nn}{\mathbb {N}}
\newcommand{\ee}{\mathbb{E}}

\newcommand{\FF}{\mathcal{F}}
\newcommand{\dd}{\mathfrak{D}}
\newcommand{\rr}{\mathbb{R}}

\newcommand{\ii}{\mathbb{I}}

\def\beq{\begin{equation}}
\def\deq{\end{equation}}
\allowdisplaybreaks %%%% 锟斤拷锟叫癸拷式锟皆讹拷锟斤拷页

\bfseries\rmfamily %%锟接达拷

\begin{document}
\title[ P\'olya-Friedman Mixed Urn Model ]
{Central Limit Theorem for a  P\'olya-Friedman Mixed Urn Model}
\thanks{This work is supported by National Natural Science Foundation of China (NSFC-11971154)
and Natural Science Foundation of Henan (No. 262300421307).}

\author[J. N. Shi]{Jianan Shi}
\address[J. N. Shi]{School of Mathematics and Statistics, Henan Normal University, Henan Province, 453007, China.}
\email{\href{mailto: J. N. Shi
<jiananshi2022@126.com>}{jiananshi2022@126.com}}

\author[Q. Yin]{Qing Yin}
\address[Q. Yin] {School of Mathematics and Statistics, Henan Normal University, Henan Province, 453007, China.}
\email{\href{mailto: Q. Yin
<qingyin1282@163.com>}{qingyin1282@163.com}}

\author[Y. Miao]{Yu Miao}
\address[Y. Miao]{School of Mathematics and Statistics, Henan Normal University, Henan Province, 453007, China.} \email{\href{mailto: Y. Miao
<yumiao728@gmail.com>}{yumiao728@gmail.com}; \href{mailto: Y. Miao <yumiao728@126.com>}{yumiao728@126.com}}

\begin{abstract}
This paper considers a two-color, single-draw urn model with two types of balls, denoted type $1$ and type $2$, with initial counts $Y^1_0\in \nn^+$ and $Y^2_0\in \nn^+$, respectively.  At each discrete time step, a ball is drawn uniformly at random, its type observed, and then it is returned to the urn. The urn is subsequently updated according to a mixed replacement matrix: with fixed probability $p\in(0,1)$, the Friedman replacement matrix is applied,
adding $a$ balls of the drawn type and $b$ balls of the opposite type;  with fixed probability
$1-p\in (0,1)$, the P\'olya replacement matrix is applied, adding $c$ balls of the drawn type. We  establish the central limit theorem for the proportion of type $1$ balls after $n$ draws. Furthermore,  we provide corollaries that yield large deviation inequalities and the law of the iterated logarithm related to the proportion of type $1$ balls after $n$ draws.
\end{abstract}


\maketitle
\section{Introduction}
Urn models have a rich history in probability theory,
with Johnson and Kotz's seminal work ``Urn Models and Their Applications"  \cite{N-S} serving as a major milestone, spurring extensive research by probabilists, statisticians, and applied scientists.
Two decades later,  Kotz and Balakrishnan \cite{S-B} published a survey paper: ``Advances in Urn Models During the Past Two Decades", which provided a thorough overview of the various types of urn models and their key characteristics.
One of the most well-known is the P\'olya urn model, introduced by Eggenberger and P\'olya \cite{E-P}. In this model, an urn initially contains
$W_0$ white balls and $B_0$ black balls.
At each discrete time step, a ball is randomly drawn from the urn, its color is recorded, and it is returned to the urn along with $c$ additional balls of the same color.
This process can be described by the $2\times2$ replacement matrix
$$
H:= \begin{pmatrix}
c & 0\\
0 & c
\end{pmatrix}
, \ \ c>0.
$$
Various generalizations of the P\'olya urn have been proposed. For instance,
Friedman \cite{F-B} considered a urn model in which $a$ balls of the same
color and $b$ balls of the opposite color are added at each step, corresponding to the symmetric replacement
matrix
$$
H:= \begin{pmatrix}
a & b\\
b & a
\end{pmatrix}
,
\ \ a, b \ge 0.
$$
Bagchi and Pal \cite{B-P} further relaxed the perfect symmetry in Friedman's urn model, considering a more general case
\beq\label{urn}
H:= \begin{pmatrix}
a & c\\
b & d
\end{pmatrix}
, \ \ a+b=c+d>0.
\deq
In this paper,  we investigate the central limit theorem for the P\'olya-Friedman mixed urn model.  The following sections will introduce the necessary definitions and developments related to urn models.

Urn models  can be roughly categorized  by the total number of balls added, the number of colors, and the number of balls drawn at each step.
An urn model is considered balanced if the total number of balls added  remains constant, regardless of the observed color, and unbalanced if it varies with color.
Urn models are further classified as  two-color or multi-color, and as single-draw or multi-draw schemes.  Gouet \cite{G-89, G-93} studied the two-color balanced urn model with single draws.
Chen and Wei \cite{C-W05} considered a two-color balanced urn model with multi-draws, where $m$ balls ($k$ black and $m-k$ white) are drawn, and returned with $c_1k$ black and $c_1(m-k)$ white balls. They showed that the proportion of black balls converges almost surely to a random variable with an absolutely continuous distribution.
Chen and Kuba \cite{C-K13} further developed this model,  providing exact formulas for the expectation and variance of white balls after $n$ draws and describing the structure of higher moments.
Dasgupta and  Maulik \cite{D-M}  considered a multi-color balanced urn model with single draws,  obtaining the rates of the counts of balls corresponding to each color for the strong laws to hold.
Idriss \cite{IS}  studied a two-color unbalanced urn model with single draws, obtained both a central limit theorem and a strong law of large numbers for the proportion of white balls at time $n$.
Shi et al. \cite{S-M,S-M-1} provided upper bounds for large deviation probabilities and the law of the iterated logarithm for these urn models.
For further results on urn models, see \cite{B-T,B-T-22, B-D-M, H-L-S, H, M-13}, among others.

In recent years, there has been an increasing amount of research on the urn model with a random replacement matrix.
Aguech et al. \cite{A-1, A-2, A-3} investigated a two-color unbalanced urn model with multi-draws and a random addition matrix,  analyzing the asymptotic behavior of the urn after $n$ draws. Crimaldi et al. \cite{C-L-M} studied a two-color unbalanced urn model with random multiple drawings and random  time-dependent addition matrix.  They proved almost sure convergence results for the proportion of balls of a given color.
Crimaldi et al. \cite{C-L-M-23}  complete the study of the model introduced in Crimaldi et al. \cite{C-L-M},  identifying  the exact rates at which the number of balls of each color grows to $+\infty$, and defined two strongly consistent estimators for the limiting reinforcement averages.

Building on these previous studies, this paper investigates a two-color, single-draw urn model with a random replacement matrix, refer to here as the mixed urn model.
In this model, upon drawing a ball and  observing its color, the system randomly selects either Friedman replacement matrix with fixed probability $p\in(0,1)$ or the P\'olya replacement matrix with fixed probability $1-p\in (0,1)$.
Alves and  Rosales \cite{A-R}  studied this mixed urn model and  showed that the proportion of balls of one color converges almost surely to $\frac{1}{2}$.
Extending their work, we focus on establishing the central limit theorem for the proportion of balls of a certain color after $n$ draws. Additionally, as a corollary, we derive large deviation inequalities and the law of the iterated logarithm pertaining to the proportion of balls of that color after $n$ draws.
The structure of this paper is as follows:
In Section 2, we present the two-color single-draw mixed urn model,  while the main results are discussed in Section 3.
Section 4 provides preliminary lemmas and proofs of the main results.

\section{A  Friedman-P\'olya mixed urn model}
We consider an urn containing two types of balls: type $1$ and type $2$. All random variables pertinent to this model are defined on a probability space $(\Omega,\mathcal{A}, \pp)$.
Let $(Y_n)=(Y^1_n,Y^2_n)^{'} \in \nn^2 \backslash \{0\},$ denote the composition of the urn at stage $n$, where $T_n:=Y^1_n+Y^2_n$ represents the total number of balls at stage $n$.
Initially, the urn contains $Y^1_0\in \nn^+$ balls of type $1$ and $Y^2_0\in \nn^+$ balls of type $2$. Let $T_0:=Y^1_0+Y^2_0$ denote the initial total number of balls.
Let $(\FF_n)_{n\ge 0}$ be the $\sigma$-algebra generated by the first $n$ steps.
At each step, a ball is drawn uniformly at random, its type observed, and then it is returned to the urn. The urn is subsequently updated according to a mixed replacement matrix: with fixed probability $p\in(0,1)$, the Friedman replacement matrix is applied,
adding $a$ balls of the drawn type and $b$ balls of the opposite type;  with fixed probability
$1-p \in(0,1)$, the P\'olya replacement matrix is applied, adding $c$ balls of the drawn type. In other words, the replacement matrix $H$ is chosen as follows:
$$
H:=
\begin{pmatrix}
a & b\\
b & a
\end{pmatrix}, \ \ \ a, b\ge 0
$$
with probability $p$, and
$$
H:=
\begin{pmatrix}
c & 0\\
0 & c
\end{pmatrix},\ \ \ c>0
$$
with probability $1-p$.

Let $Z_n:=\frac{Y^1_n}{T_n}$ represent the proportion of balls of type $1$ in the urn. In order to obtain our main results, we first need to establish the recursive formula for $Z_n$.
Let $(\xi_n)_{n\ge 1}$ be a sequence of independent and identically distributed random variables, independent of $\FF_n$, defined such that $$
\pp(\xi_n=1)=p,\ \ \ \pp(\xi_n=0)=1-p,\ \ \ p\in(0,1).
$$
Here, $\xi_n$ indicates the replacement rule applied at step $n$: $\xi_n=1$ means the Friedman replacement matrix is used, while $\xi_n=0$ means the P\'olya replacement matrix is used.
Let $(\eta_n)_{n\ge 1}$ be another sequence of random variables such that
$$
\pp(\eta_n=1|\FF_{n-1})=Z_{n-1},\ \ \ \pp(\eta_n=0|\FF_{n-1})=1-Z_{n-1}.
$$
Here, $\eta_n$ indicates the type of ball drawn at step $n$: $\eta_n=1$ corresponds to drawing a type $1$ ball, and $\eta_n=0$ corresponds to drawing a type $2$ ball. The conditional probability of drawing a type $1$ ball is exactly the current proportion $Z_{n-1}$.
In this case, we have
\beq\label{Y}
\aligned
Y^1_{n+1}&=Y^1_n+a \ii_{\{\xi_{n+1}=1,\eta_{n+1}=1\}}+b\ii_{\{\xi_{n+1}=1,\eta_{n+1}=0\}}+c \ii_{\{\xi_{n+1}=0,\eta_{n+1}=1\}}\\
&=Y^1_0+a\sum_{k=1}^{n+1} \ii_{\{\xi_k=1,\eta_k=1\}}+b\sum_{k=1}^{n+1} \ii_{\{\xi_k=1,\eta_k=0\}}+c\sum_{k=1}^{n+1} \ii_{\{\xi_k=0,\eta_k=1\}}
\endaligned
\deq
and
\beq\label{T}
\aligned
T_{n+1}&=T_n+(a+b)\ii_{\{\xi_{n+1}=1\}}+c\ii_{\{\xi_{n+1}=0\}}\\
&=T_0+\sum_{k=1}^{n+1} \left[(a+b)\ii_{\{\xi_k=1\}}+c\ii_{\{\xi_k=0\}}\right].
\endaligned
\deq
Define
\beq\label{Y1}
\Delta Y^1_{n+1}:=a \ii_{\{\xi_{n+1}=1,\eta_{n+1}=1\}}+b\ii_{\{\xi_{n+1}=1,\eta_{n+1}=0\}}+c \ii_{\{\xi_{n+1}=0,\eta_{n+1}=1\}}
\deq
and
\beq\label{T1}
\Delta T_{n+1}:=(a+b)\ii_{\{\xi_{n+1}=1\}}+c\ii_{\{\xi_{n+1}=0\}}.
\deq
Then for $Z_n:=\frac{Y^1_n}{T_n}$ with $Z_0\in(0,1)$, we have
\beq\label{Z}
\aligned
Z_{n+1}-Z_n&=\frac{Y_{n+1}^1}{T_{n+1}}-\frac{Y_{n}^1}{T_{n}}
=\frac{Y_{n}^1+\Delta Y_{n+1}^1}{T_n+\Delta T_{n+1}}-\frac{Y_{n}^1}{T_{n}}\\
&=\frac{T_n\left(Y_{n}^1+\Delta Y_{n+1}^1\right)-Y_{n}^1\left(T_n+\Delta T_{n+1}\right)}{T_n\left(T_n+\Delta T_{n+1}\right)}\\
&=\frac{T_n\Delta Y_{n+1}^1-Y_{n}^1\Delta T_{n+1}}{T_n\left(T_n+\Delta T_{n+1}\right)}\\
&=\frac{\Delta Y_{n+1}^1}{T_{n+1}}-\frac{Y_{n}^1\Delta T_{n+1}}{T_n T_{n+1}}\\
&=\frac{\Delta Y_{n+1}^1}{T_{n+1}}-\frac{\Delta T_{n+1}}{ T_{n+1}}Z_n.
\endaligned
\deq
Note that $(\xi_n)_{n\ge 1}$ be a sequence of random variables, independent of $\FF_n$, from (\ref{Y1}) we can get that
$$
\aligned
\ee\left[\Delta Y_{n+1}^1|\FF_n\right]&=\ee\left[a \ii_{\{\xi_{n+1}=1,\eta_{n+1}=1\}}+b\ii_{\{\xi_{n+1}=1,\eta_{n+1}=0\}}+c \ii_{\{\xi_{n+1}=0,\eta_{n+1}=1\}}|\FF_n\right]\\
&=a\pp\left(\xi_{n+1}=1,\eta_{n+1}=1|\FF_n\right)+b\pp\left(\xi_{n+1}=1,\eta_{n+1}=0|\FF_n\right)\\
&\ \ \ \ \ \ \ \ \ \ \ \ \  +c\pp\left(\xi_{n+1}=0,\eta_{n+1}=1|\FF_n\right)\\
&=a p Z_{n}+b p (1-Z_{n})+c(1-p)Z_{n},
\endaligned
$$
and from (\ref{T1}) we can get that
$$
\aligned
\ee\left[\Delta T_{n+1}Z_n|\FF_n\right]&=Z_n\ee\left[(a+b)\ii_{\{\xi_{n+1}=1\}}+c\ii_{\{\xi_{n+1}=0\}}|\FF_n\right]\\
&=Z_n(a+b)\pp\left(\xi_{n+1}=1|\FF_n\right)+Z_nc\pp\left(\xi_{n+1}=0|\FF_n\right)\\
&=(a+b)pZ_n+c(1-p)Z_n.
\endaligned
$$
Hence we can get that
\beq\label{M1}
\ee\left[(\Delta Y_{n+1}^1-\Delta T_{n+1}Z_n)\Big|\FF_n\right]=-2b pZ_n+b p.
\deq
Define
\beq\label{f}
f(x):=-2b px+b p
\deq
and
\beq\label{M}
\aligned
\Delta M_{n+1}:&=\Delta Y_{n+1}^1-\Delta T_{n+1}Z_n-\ee\left[\left(\Delta Y_{n+1}^1-\Delta T_{n+1}Z_n\right)\Big|\FF_n\right]\\
&=\Delta Y_{n+1}^1-\Delta T_{n+1}Z_n-f(Z_n),
\endaligned
\deq
from (\ref{Z}) we know that
\beq\label{Z-1}
Z_{n+1}-Z_n=\frac{1}{T_{n+1}}\left(f(Z_n)+\Delta M_{n+1}\right).
\deq

\section{Main results}
In this section, we present the central limit theorem for the proportion $(Z_n)_{n\ge0}$ of balls of type $1$ in the two-color, P\'olya-Friedman mixed urn model with single draws.

\begin{thm}\label{thm1}
Let $Y^1_0$ and $Y^2_0$ be non-negative integers such that $T_0:=Y^1_0+Y^2_0>0$. For any positive integers $a,b,c$ and any $p\in (0,1)$,
we can establish the following properties for the total number of balls $T_n$ after $n$ draws,  along with the function $f(x)$ and the increment $\Delta M_{n+1}$ as defined in  (\ref{f}) and (\ref{M}):
\begin{enumerate}[$(i)$]
 \item For any $n\ge 1$, we have
  \beq\label{T2}
\frac{1}{T_0+n\max\{a+b,c\}}\le\frac{1}{T_n}\le \frac{1}{n\min\{a+b,c\}}.
\deq

\item For any $x\in [0,1]$, we have
\beq
\left|f(x)\right|\le b.
\deq

\item
For any $n\ge 0$, we have
\beq
|\Delta M_{n+1}|\le 2a+3b+2c.
\deq

 \item
 For any $n\ge 1$, we have
\beq\label{V-2}
\left|\ee\left[\frac{\Delta M_{n+1}}{T_{n+1}}\Big|\FF_n\right]\right|\le \frac{K}{n^2}.
\deq
\end{enumerate}
Here, $K$  is a positive constant.
\end{thm}

\begin{cor}
 Based on  Theorem \ref{thm1}, we can conclude that the sequence $(Z_n)_{n\ge 0}$, as defined in (\ref{Z-1}), meets the criteria specified in Theorem 4.1 of Shi et al. \cite{S-M}. Consequently, we can derive large deviation inequalities for $(Z_n)_{n\ge 0}$.  Specifically,  for any $\varepsilon > 0$, there exists a positive constant $a$, which is independent of $n$, such that for all sufficiently large $n$, we have
$$
\pp\left(\left|Z_{n+1}-\frac{1}{2}\right|>\varepsilon\right)\le 2 e^{-a n}.
$$
\end{cor}

\begin{thm}\label{thm2}
Let $Y^1_0$ and $Y^2_0$ be non-negative integers such that $T_0:=Y^1_0+Y^2_0>0$. For any positive integers $a,b,c$ and any $p\in (0,1)$,
the total number of balls  after $n$ draws, denoted as $T_n$, satisfies
$$
\lim_{n\rightarrow\infty} \frac{T_n}{n}= (a+b)p+c(1-p).
$$
Moreover, we have the convergence for the conditional expectation:
$$
\ee\left[\left(\frac{n+1}{T_{n+1}} \Delta M_{n+1}\right)^2\Bigg|\FF_{n} \right]\rightarrow \frac{p(a-b)^2+c^2(1-p)}{4\big((a+b)p+c(1-p)\big)^2}\ \ \ \ a.s.,
$$
where $\Delta M_{n+1}$ is defined in (\ref{M}).
\end{thm}

\begin{cor}
 Building on Theorem \ref{thm2}, we establish that the sequence $(Z_n)_{n\ge 0}$, defined in (\ref{Z-1}), fulfills the criteria outlined in  Theorem 3.1 of Shi et al. \cite{S-M-1}. As a result, we can derive the law of the iterated logarithm for $(Z_n)_{n\ge 0}$. Under the condition that $\frac{-b p}{(a+b)p+c(1-p)}<-\frac{1}{2}$, the following limit holds
$$
\limsup_{n\rightarrow \infty} \left(\frac{n}{2 \log \log n}\right)^{1/2} \left(Z_{n+1}-\frac{1}{2}\right)=\frac{\sigma}{(2 \alpha +1)^{\frac{1}{2}}},
$$
where $\alpha:=\frac{b p}{(a+b)p+c(1-p)}-1$,  and
$
\sigma^2:=\frac{p(a-b)^2+c^2(1-p)}{4\big((a+b)p+c(1-p)\big)^2}.
$
\end{cor}

\begin{thm}\label{thm2-1}
Let $Y^1_0$ and $Y^2_0$ be non-negative integers such that $T_0:=Y^1_0+Y^2_0>0$.  Define the proportion of balls of type 1 in the urn as $Z_n:=\frac{Y^1_n}{T_n}$, as indicated in (\ref{Z-1}).
For any positive integers $a,b,c$ and any $p\in (0,1)$,  if the condition $3 b p+c p>a p+c$ is satisfied, we have
\beq\label{CLT}
\sqrt{n} \left(Z_n-\frac{1}{2}\right)\stackrel{\dd}{\longrightarrow} \nn\left(0,\frac{\sigma^2}{2\Gamma-1}\right),
\deq
where
$$
\sigma^2:=\frac{p(a-b)^2+c^2(1-p)}{4\big((a+b)p+c(1-p)\big)^2}
$$
and
$$
\Gamma:= \frac{2b p}{(a+b)p+c(1-p)}.
$$
\end{thm}

\begin{rem}
The condition $3 b p+c p>a p+c$ ensures that $2\Gamma-1>0$, making the result meaningful.
\end{rem}

\begin{rem}\label{rem-1}
  Under these conditions, Alves and  Rosales \cite{A-R}  showed that $$\lim_{n\rightarrow \infty } Z_n=\frac{1}{2}\ \ a.s. $$
\end{rem}

\begin{exa}
Consider the case where $a=1$, $b=3$, $c=2$, and $p=\frac{1}{4}$, it is easy to verify that $3 b p+c p>a p+c$ holds. In this scenario, we have
$$
\sigma^2:=\frac{p(a-b)^2+c^2(1-p)}{4\big((a+b)p+c(1-p)\big)^2}=\frac{\frac{1}{4}\times  4+4 \times \frac{3}{4}}{4\left(1+\frac{3}{2} \right)^2} =\frac{4}{25}
$$
and
$$
\Gamma:= \frac{2b p}{(a+b)p+c(1-p)}=\frac{3}{5}.
$$
Since  we can conclude that
$$
\sqrt{n} \left(Z_n-\frac{1}{2}\right)\stackrel{\dd}{\longrightarrow} \nn\left(0,\frac{4}{5}\right).
$$
\end{exa}

\section{Proofs of main results}
Throughout this section, we shall use some notation which has been defined in Section 2, for example, $(\Delta M_{n})_{n\ge 1}$, $(Z_n)_{n\ge 0}$, $(T_n)_{n\ge 1}$, and so on.

\begin{proof}[{\bf Proof of Theorem \ref{thm1}}]
\begin{enumerate}[(i)]
\item
From (\ref{T}), we know that
$$
T_n+\min\{a+b,c\}\le T_{n+1}\le T_n+\max\{a+b,c\},
$$
hence we have
\beq\label{T-2}
\frac{1}{T_0+n\max\{a+b,c\}}\le\frac{1}{T_n}\le \frac{1}{n\min\{a+b,c\}}.
\deq

\item
For any $x\in [0,1]$, note that $p\in(0,1)$, it follows that
$$
\left|f(x)\right|=\left|b p\left(1-2x\right)\right|\le b p< b.
$$

\item
From (\ref{Y1}), (\ref{T1}) and (\ref{M}), we have
\begin{align*}
|\Delta M_{n+1}|&=\left|\Delta Y_{n+1}^1-\Delta T_{n+1}Z_n-f(Z_n)\right|\\
&\le \left|\Delta Y_{n+1}^1\right|+\left|\Delta T_{n+1}Z_n\right|+\left|f(Z_n)\right|\\
&\le 2(a+b+c)+b=2a+3b+2c.
\end{align*}

\item
From (\ref{T}) and (\ref{M}), it is not difficult to see that
\begin{align*}
\frac{\Delta M_{n+1}}{T_{n+1}}&=\frac{\Delta Y_{n+1}^1-\Delta T_{n+1}Z_n-f(Z_n)}{T_n+(a+b)\ii_{\{\xi_{n+1}=1\}}+c\ii_{\{\xi_{n+1}=0\}}}\\
&=\frac{a \ii_{\{\xi_{n+1}=1,\eta_{n+1}=1\}}+b\ii_{\{\xi_{n+1}=1,\eta_{n+1}=0\}}+c \ii_{\{\xi_{n+1}=0,\eta_{n+1}=1\}}}{T_n+(a+b)\ii_{\{\xi_{n+1}=1\}}+c\ii_{\{\xi_{n+1}=0\}}}\\
&\ \ \ \ \ \ \ \ -\frac{(a+b)\ii_{\{\xi_{n+1}=1\}}+c\ii_{\{\xi_{n+1}=0\}}}{T_n+(a+b)\ii_{\{\xi_{n+1}=1\}}+c\ii_{\{\xi_{n+1}=0\}}}Z_n\\
&\ \ \ \ \ \ \ \ -\frac{-2b p Z_n+b p}{T_n+(a+b)\ii_{\{\xi_{n+1}=1\}}+c\ii_{\{\xi_{n+1}=0\}}}\\
&=\frac{a \ii_{\{\eta_{n+1}=1\}}+b\ii_{\{\eta_{n+1}=0\}}}{T_n+a+b}\ii_{\{\xi_{n+1}=1\}}+\frac{c \ii_{\{\eta_{n+1}=1\}}}{T_n+c}\ii_{\{\xi_{n+1}=0\}}\\
&\ \ \ \ \ \ \ \ -\frac{aZ_n+bZ_n-2b p Z_n+b p}{T_n+a+b}\ii_{\{\xi_{n+1}=1\}}-\frac{cZ_n-2b p Z_n+b p}{T_n+c}\ii_{\{\xi_{n+1}=0\}}\\
&=\frac{a \ii_{\{\eta_{n+1}=1\}}+b\ii_{\{\eta_{n+1}=0\}}-\left(aZ_n+bZ_n-2b p Z_n+b p\right)}{T_n+a+b}\ii_{\{\xi_{n+1}=1\}}\\
&\ \ \ \ \ \ \ \ +\frac{c \ii_{\{\eta_{n+1}=1\}}-\left(cZ_n-2b p Z_n+b p\right)}{T_n+c}\ii_{\{\xi_{n+1}=0\}}.
\end{align*}
Hence we can get that
\begin{align*}
\ee\left[\frac{\Delta M_{n+1}}{T_{n+1}}\Big|\FF_n\right]&=\ee\left[\frac{a \ii_{\{\eta_{n+1}=1\}}+b\ii_{\{\eta_{n+1}=0\}}-(aZ_n+bZ_n-2b p Z_n+b p)}{T_n+a+b}\ii_{\{\xi_{n+1}=1\}}\Bigg|\FF_n\right]\\
&\ \ \ \ \ \ \ \ +\ee\left[\frac{c \ii_{\{\eta_{n+1}=1\}}-(cZ_n-2b p Z_n+b p)}{T_n+c}\ii_{\{\xi_{n+1}=0\}}\Bigg|\FF_n\right]\\
&=\frac{a Z_n+b(1-Z_n)-aZ_n-bZ_n+2b p Z_n-b p}{T_n+a+b}p\\
&\ \ \ \ \ \ \ \ \ +\frac{c Z_n-cZ_n+2b p Z_n-b p}{T_n+c}(1-p)\\
&=\frac{b-2bZ_n+2b p Z_n-b p}{T_n+a+b}p+\frac{2b p Z_n-b p}{T_n+c}(1-p)\\
&=\frac{bp(1-p)-2bpZ_n(1-p)}{T_n+a+b}+\frac{2b p Z_n(1-p)-b p(1-p)}{T_n+c}\\
&=bp(1-p)\left(\frac{1}{T_n+a+b}-\frac{1}{T_n+c}\right)\\
&\ \ \ \ \ \ \ \ +2bpZ_n(1-p)\left(\frac{1}{T_n+c}-\frac{1}{T_n+a+b}\right)\\
&=\left(bp(1-p)-2bpZ_n(1-p)\right)\left(\frac{1}{T_n+a+b}-\frac{1}{T_n+c}\right)\\
&=bp(1-p)(1-2Z_n)\frac{c-a-b}{(T_n+a+b)(T_n+c)}.
\end{align*}
From (\ref{T-2}), we have
$$
\left|\ee\left[\frac{\Delta M_{n+1}}{T_{n+1}}\Big|\FF_n\right]\right|\le \frac{K}{n^2},
$$
where $K$ is a positive constant.
\end{enumerate}
\end{proof}

\begin{proof}[{\bf Proof of Theorem \ref{thm2}}]
From (\ref{T}),
$$
\frac{T_n}{n}=\frac{T_0+\sum_{k=1}^{n} \left[(a+b)\ii_{\{\xi_k=1\}}+c\ii_{\{\xi_k=0\}}\right]}{n}.
$$
Since $(\xi_n)_{n\ge 1}$ is a sequence of independent and identically distributed random variables with
$\pp(\xi_n=1)=p$ and $\pp(\xi_n=0)=1-p$,
the law of large numbers  implies that as $n\rightarrow\infty$,
\beq\label{T3}
\frac{T_n}{n}\rightarrow (a+b)p+c(1-p)\ \ a.s.
\deq
From (\ref{M}), it is not difficult to see that
$$
\aligned
\ee\left[\left(\Delta M_{n+1}\right)^2\Big|\FF_n\right]&=\ee\left[\left(\Delta Y_{n+1}^1-\Delta T_{n+1}Z_n-\ee\left[\left(\Delta Y_{n+1}^1-\Delta T_{n+1}Z_n\right)\Big|\FF_n\right]\right)^2\Big|\FF_n\right]\\
&=\ee\left[\left(\Delta Y_{n+1}^1-\Delta T_{n+1}Z_n\right)^2\Big|\FF_n\right]-\ee\left[\left(\Delta Y_{n+1}^1-\Delta T_{n+1}Z_n\right)\Big|\FF_n\right]^2.
\endaligned
$$
From (\ref{M1}), it follows that
$$
\ee\left[\left(\Delta Y_{n+1}^1-\Delta T_{n+1}Z_n\right)\Big|\FF_n\right]^2=\left(-2b pZ_n+b p\right)^2.
$$
Therefore, we only need to handle item $\ee\left[\left(\Delta Y_{n+1}^1-\Delta T_{n+1}Z_n\right)^2\Big|\FF_n\right]$. From (\ref{Y1}), we know that
$$
\aligned
\ee\left[\left(\Delta Y_{n+1}^1\right)^2\Big|\FF_n\right]&=\ee\left[\left(a \ii_{\{\xi_{n+1}=1,\eta_{n+1}=1\}}+b\ii_{\{\xi_{n+1}=1,\eta_{n+1}=0\}}+c \ii_{\{\xi_{n+1}=0,\eta_{n+1}=1\}}\right)^2\Big|\FF_n\right]\\
&=\ee\left[a^2 \ii_{\{\xi_{n+1}=1,\eta_{n+1}=1\}}+b^2\ii_{\{\xi_{n+1}=1,\eta_{n+1}=0\}}+c^2 \ii_{\{\xi_{n+1}=0,\eta_{n+1}=1\}}\Big|\FF_n\right]\\
&=a^2p Z_n+b^2p (1-Z_n)+c^2(1-p)Z_n.
\endaligned
$$
From (\ref{T1}),
$$
\aligned
\ee\left[\left(\Delta T_{n+1}Z_n\right)^2\Big|\FF_n\right]&=\ee\left[Z_n^2\left((a+b)\ii_{\{\xi_{n+1}=1\}}+c\ii_{\{\xi_{n+1}=0\}}\right)^2\Big|\FF_n\right]\\
&=Z_n^2\ee\left[(a+b)^2\ii_{\{\xi_{n+1}=1\}}+c^2\ii_{\{\xi_{n+1}=0\}}\Big|\FF_n\right]\\
&=Z_n^2(a+b)^2p+Z_n^2c^2(1-p).
\endaligned
$$
Observe that
$$
\aligned
&\ee\left[\Delta Y_{n+1}^1\Delta T_{n+1}Z_n\Big|\FF_n\right]=Z_n\ee\left[\Delta Y_{n+1}^1\Delta T_{n+1}\Big|\FF_n\right]\\
&=Z_n\ee\left[a(a+b) \ii_{\{\xi_{n+1}=1,\eta_{n+1}=1\}}+b(a+b)\ii_{\{\xi_{n+1}=1,\eta_{n+1}=0\}}+c^2 \ii_{\{\xi_{n+1}=0,\eta_{n+1}=1\}}\Big|\FF_n\right]\\
&=Z_n\left(a(a+b)p Z_n+b(a+b)p (1-Z_n)+c^2(1-p)Z_n\right)\\
&=(a+b)pZ_n\left(aZ_n+b(1-Z_n)\right)+c^2(1-p)Z_n^2.
\endaligned
$$
Hence we can get that
$$
\aligned
&\ee\left[\left(\Delta Y_{n+1}^1-\Delta T_{n+1}Z_n\right)^2\Big|\FF_n\right]\\
&=\ee\left[\left(\Delta Y_{n+1}^1\right)^2\Big|\FF_n\right]+\ee\left[\left(\Delta T_{n+1}Z_n\right)^2\Big|\FF_n\right]-2\ee\left[\Delta Y_{n+1}^1\Delta T_{n+1}Z_n\Big|\FF_n\right]\\
&=a^2p Z_n+b^2p (1-Z_n)+c^2(1-p)Z_n+Z_n^2(a+b)^2p\\
&\ \ \ \ \ \ \ \ \ \ -2(a+b)pZ_n\left(aZ_n+b(1-Z_n)\right)-c^2(1-p)Z_n^2.
\endaligned
$$
From Remark \ref{rem-1}, we know that $\displaystyle\lim_{n\rightarrow\infty }Z_n=\frac{1}{2}$, hence we can get that as $n\rightarrow\infty$,
$$
\ee\left[\left(\Delta Y_{n+1}^1-\Delta T_{n+1}Z_n\right)^2\Big|\FF_n\right]\rightarrow \frac{1}{4}\left(p(a-b)^2+c^2(1-p)\right)
$$
and
$$
\ee\left[\left(\Delta Y_{n+1}^1-\Delta T_{n+1}Z_n\right)\Big|\FF_n\right]^2=\left(-2b pZ_n+b p\right)^2\rightarrow0.
$$
Together with (\ref{T3}),
$$
\ee\left[\left(\frac{n+1}{T_{n+1}} \Delta M_{n+1}\right)^2\Bigg|\FF_{n} \right]\rightarrow \frac{p(a-b)^2+c^2(1-p)}{4\big((a+b)p+c(1-p)\big)^2}\ \ \ a.s.
$$
\end{proof}

\begin{lem}\label{cor-1}
Suppose that $(h_n)_{n\ge 1}$ is a sequence of real number such that
$$
h_{n+1}=\left(1-\frac{L_n}{n+1}\right) h_n+\frac{Q_n}{n+1},
$$
where $\displaystyle\lim_{n\rightarrow\infty} L_n=L_0>0$ and $\displaystyle\lim_{n\rightarrow\infty} Q_n=Q_0\ge0$, then we have
$$
\displaystyle\lim_{n\rightarrow\infty} h_n=\frac{Q_0}{L_0}.
$$
\end{lem}

\begin{proof} Let
$$
p_n:= h_n-\frac{Q_0}{L_0}.
$$
Hence we can get that
$$
p_{n+1}+\frac{Q_0}{L_0}=\left(1-\frac{L_n}{n+1}\right) \left(p_n+\frac{Q_0}{L_0}\right)+\frac{Q_n}{n+1}.
$$
After a simple calculation, it can be determined that
$$
p_{n+1}=\left(1-\frac{L_n}{n+1}\right)p_n+\frac{s_n}{n+1},
$$
where $s_n:=Q_n-L_n \frac{Q_0}{L_0}$.
Since $\displaystyle\lim_{n\rightarrow\infty} L_n=L_0>0$ and $\displaystyle\lim_{n\rightarrow\infty} Q_n=Q_0\ge0$, there exist constants $N_0$ and $a_0$ such that for all $n>N_0$, we have $L_n>a_0>0$. Moreover, since $s_n\rightarrow 0$, there exists a sequence $\varepsilon_n\rightarrow0$ such that $|s_n|\le \varepsilon_n $ for all $n>N_0$.
Define $d_n:=|p_n|$, then for sufficiently large $n$,
$$
d_{n+1}\le \left(1-\frac{a_0}{n+1}\right)d_n+\frac{\varepsilon_n}{n+1}.
$$
Assume that $\displaystyle\limsup_{n\rightarrow\infty} d_n=A_0>0$, there exists a subsequence $(n_k)_{k\ge 0}$ such that $\displaystyle\lim_{k\rightarrow\infty} d_{n_k}=A_0$. Choose $\delta>0$ small enough and $K_0$ large enough so that for all $k>K_0$, $d_{n_k}>A_0-\delta$ and $\varepsilon_{n_k}<\frac{a_0 (A_0-\delta)}{2}$. Hence for all $k>K_0$, we have
$$
\aligned
d_{n_k+1}&\le \left(1-\frac{a_0}{n_k+1}\right)d_{n_k}+\frac{\varepsilon_{n_k}}{n_k+1}\\
&\le d_{n_k}- \frac{a_0\left(A_0-\delta\right)}{n_k+1}+\frac{a_0 (A_0-\delta)}{2\left(n_k+1\right)}\\
&\le d_{n_k}-\frac{a_0 (A_0-\delta)}{2(n_k+1)}.
\endaligned
$$
Iterating this bound gives that for any $m\ge 1$,
$$
d_{n_k+m}\le  d_{n_k}-\frac{a_0 \left(A_0-\delta\right)}{2}\sum_{i=n_k}^{n_k+m-1} \frac{1}{i+1}.
$$
Since $\sum_{i=n_k}^\infty \frac{1}{i+1}=\infty$, the right-hand side tends to
$-\infty$ as $m\rightarrow\infty$, which contradicts $d_{n}$ being non-negative.
Therefore, it is necessary that
$$
\limsup_{n\rightarrow\infty} d_n=0,
$$
which implies
$$
 \displaystyle\lim_{n\rightarrow\infty} h_n=\frac{Q_0}{L_0}.
$$
\end{proof}

\begin{cor}\label{lem-1}
Let $(X_n)_{n\ge 0}$ be a stochastic process taking values in $\rr$ that satisfies the recursion
$$
X_{n+1}=\left(1-\frac{a}{n+1}\right)X_n+\frac{K_{n+1}}{\sqrt{n+1}},
$$
where $a>0$ is a constant, and $(K_n)_{n\ge 1}$ is a sequence of independent and identically distributed random variables with distribution $\nn(0,\sigma^2)$, then we have
$$
\lim_{n\rightarrow\infty}\ee X_{n+1}^2= \frac{\sigma^2}{2 a}.
$$
\end{cor}

\begin{proof}
Squaring both sides of the recursion gives
$$
X_{n+1}^2=\left(1-\frac{a}{n+1}\right)^2X_n^2+\frac{K_{n+1}^2}{n+1}+2\left(1-\frac{a}{n+1}\right)X_n\frac{K_{n+1}}{\sqrt{n+1}}.
$$
Taking the expectation on both sides and using the facts that $(K_n)_{n\ge 1}$ is a sequence of independent and identically distributed  random variables with $\ee K_n=0$ and $\ee K_n^2=\sigma^2$, and setting $b_n:=\ee X_n^2$, we obtain
$$
b_{n+1}=\left(1-\frac{a}{n+1}\right)^2 b_n+\frac{\sigma^2}{n+1}=\left(1-\frac{A_n}{n+1}\right) b_n+\frac{\sigma^2}{n+1},
$$
where $A_n:=2 a -\frac{a^2}{n+1}\rightarrow 2a, n\rightarrow\infty$. Define $c_n:=b_n-\frac{\sigma^2}{2 a}$, hence
$$
c_{n+1}+\frac{\sigma^2}{2 a}=\left(1-\frac{A_n}{n+1}\right)\left(c_n+\frac{\sigma^2}{2 a}\right)+\frac{\sigma^2}{n+1}.
$$
Through simple calculation, we obtain
\beq\label{c}
c_{n+1}=\left(1-\frac{A_n}{n+1}\right) c_n+\frac{\gamma_n}{n+1},
\deq
where $\gamma_n:=\sigma^2-\frac{A_n \sigma^2}{2 a}$, and note that $\displaystyle\lim_{n\rightarrow\infty} A_n=2 a>0$ and $\displaystyle\lim_{n\rightarrow\infty} \gamma_n=0$. From Lemma \ref{cor-1} we can get that
$$
\displaystyle\lim_{n\rightarrow\infty} c_n=0,
$$
which implies
$$
 \lim_{n\rightarrow\infty} \ee X_{n}^2= \lim_{n\rightarrow\infty} b_n=\frac{\sigma^2}{2 a}.
$$
\end{proof}

\begin{lem}\label{lem-2}\cite[Page 556, Lemma 1.2]{G}
For any $n\ge 0$,
$$
\left|e^{i y}-\sum_{k=0}^n \frac{(i y)^k}{k!}\right|\le \min\left\{2 \frac{|y|^n}{n!}, \frac{|y|^{n+1}}{(n+1)!}\right\}.
$$
\end{lem}

\begin{lem} (Slutsky) \label{lem-3}\cite[Page 19]{S}
Let $X_n \stackrel{d}{\longrightarrow} X$ and $Y_n \stackrel{P}{\longrightarrow} c$, where $c$ is a finite constant. Then
\begin{enumerate}[$(i)$]
 \item $X_n+Y_n\stackrel{d}{\longrightarrow} X+c$;

 \item $X_n Y_n\stackrel{d}{\longrightarrow} c X$;

 \item $\frac{X_n}{Y_n} \stackrel{d}{\longrightarrow} \frac{X}{c}$ if $c\neq 0$.

\end{enumerate}
\end{lem}

\begin{proof}[{\bf Proof of Theorem \ref{thm2-1}}]
From (\ref{Z-1}) and (\ref{f}),
$$
\aligned
Z_{n+1}-\frac{1}{2}&=Z_n-\frac{1}{2}+\frac{1}{T_{n+1}}\left(f(Z_n)+\Delta M_{n+1}\right)\\
&=Z_n-\frac{1}{2}+\frac{-2b p}{T_{n+1}}\left(Z_n-\frac{1}{2}\right)+\frac{\Delta M_{n+1}}{T_{n+1}}\\
&=\left(Z_n-\frac{1}{2}\right)\left(1-\frac{2b p}{T_{n+1}}\right)+\frac{\Delta M_{n+1}}{T_{n+1}}.
\endaligned
$$
Hence we have
$$
\sqrt{n+1}\left(Z_{n+1}-\frac{1}{2}\right)=\sqrt{n+1}\left(Z_n-\frac{1}{2}\right)\left(1-\frac{2b p}{T_{n+1}}\right)+\frac{\sqrt{n+1}\Delta M_{n+1}}{T_{n+1}}.
$$
Defined that
$$
X_{n+1}:=\sqrt{n+1}\left(Z_{n+1}-\frac{1}{2}\right),\ \ \Gamma_{n+1}:=\frac{2b p(n+1)}{T_{n+1}}\ \,\ \ \text{and}\ \ V_{n+1}:=\frac{(n+1)\Delta M_{n+1}}{T_{n+1}},
$$
it is not difficult to get that
$$
X_{n+1}=\sqrt{1+\frac{1}{n}}\left(1-\frac{\Gamma_{n+1}}{n+1}\right)X_n+\frac{V_{n+1}}{\sqrt{n+1}}.
$$
Through Taylor expansion, we obtain
$$
\sqrt{1+\frac{1}{n}}=1+\frac{1}{2(n+1)}+O\left(\frac{1}{(n+1)^2}\right).
$$
Thus,
\beq\label{X-1}
X_{n+1}=\left(1-\frac{\Gamma_{n+1}-\frac{1}{2}+O\left(\frac{1}{n+1}\right)}{n+1}\right)X_n+\frac{V_{n+1}}{\sqrt{n+1}}.
\deq
From Theorem \ref{thm1},
\beq\label{V}
V_{n+1}^2=\left(\frac{(n+1)\Delta M_{n+1}}{T_{n+1}}\right)^2\le \frac{(2a+3b+2c)^2}{(\min\{a+b,c\})^2}:=C_V.
\deq
From Theorem \ref{thm2}, we can get that
\beq\label{T4}
\Gamma_{n+1}-\frac{1}{2}+O\left(\frac{1}{n+1}\right)\rightarrow \frac{2b p}{(a+b)p+c(1-p)}-\frac{1}{2}:=\Gamma-\frac{1}{2}>0
\deq
and
\beq\label{V-1}
\ee\left[V_{n+1}^2|\FF_n\right]=\ee\left[\left(\frac{n+1}{T_{n+1}} \Delta M_{n+1}\right)^2\Bigg|\FF_{n} \right]\rightarrow \frac{p(a-b)^2+c^2(1-p)}{4\big((a+b)p+c(1-p)\big)^2}:=\sigma^2\ \ a.s.
\deq
Therefore, to prove (\ref{CLT}), it suffices to prove
\beq\label{p-1}
X_{n+1} \stackrel{\dd}{\longrightarrow} \nn\left(0,\frac{\sigma^2}{2\Gamma-1}\right).
\deq
We will prove this in two steps.

{\bf Step 1:} Assuming that  $\Gamma_{n+1}-\frac{1}{2}+O\left(\frac{1}{n+1}\right)\equiv \Gamma-\frac{1}{2}$ and $\ee\left[V_{n+1}|\FF_n\right]=0$,  we first prove that (\ref{p-1}) holds under these conditions.

{\bf Step 2:} We then demonstrate that these assumptions can be removed without affecting the validity of the result, completing the proof.

Below, we provide a detailed proof of these two steps.

{\bf Step 1. Detailed Proof:} In this step we work under the following simplifying assumptions (the general case will be handled by error estimates in Step 2):
$$ \Gamma_{n+1}-\frac{1}{2}+O\left(\frac{1}{n+1}\right)\equiv \Gamma-\frac{1}{2}>0,\ \ \ \ \ \ee\left[V_{n+1}|\FF_n\right]=0.
$$
Define
$$
\varphi_n(t):=\ee e^{it X_n}, \ \ B_n:=1-\frac{\Gamma-\frac{1}{2}}{n+1}.
$$
From (\ref{X-1}), we obtain the exact recursion
\beq\label{X}
X_{n+1}=B_n X_n+\frac{V_{n+1}}{\sqrt{n+1}}.
\deq
To prove (\ref{p-1}), we construct an auxiliary sequence $\left(\beta_n(t)\right)_{n\ge 1}$ as follows:
\beq\label{xi}
\beta_1(t)=1,\ \ \beta_{n+1}(t):=\beta_n\left(B_nt\right)\left(1-\frac{t^2 \sigma^2}{2(n+1)}\right).
\deq
It can be seen that $\left(\beta_n(t)\right)_{n\ge 1}$  is precisely the characteristic function of $(X_{n})_{n\ge 1}$ when $\left(V_n\right)_{n\ge 1}$ is a sequence of independent and identically distributed random variables with distribution $\nn(0,\sigma^2)$. Then, by Corollary \ref{lem-1},  in this normal case we have $\ee X_n^2 \rightarrow \frac{\sigma^2}{2\Gamma-1}$. Since each
$X_n$ is normal with mean $\mu_n=\ee X_0 \prod_{k=0}^{n-1} B_k$,  its characteristic function is $\beta_{n}(t)=e^{i\mu_n t-\frac{1}{2}t^2\ee X_n^2}$. Note that $\mu_n\sim \ee X_0 n^{-\left(\Gamma-\frac{1}{2}\right)}\rightarrow 0$ and $\ee X_n^2 \rightarrow \frac{\sigma^2}{2\Gamma-1}$,
we obtain that for all $t$,
$$
\beta_{n}(t)\rightarrow e^{-\frac{1}{2}t^2 \frac{\sigma^2}{2\Gamma-1}}.
$$
If we can prove that
\beq\label{p-2}
\left|\varphi_n(t)- \beta_{n}(t)\right|\rightarrow 0,\ \ \text{as}\ \ n\rightarrow\infty,\ \ \text{for all}\ \ t,
\deq
then we obtain
$$
\varphi_n(t)\rightarrow e^{-\frac{1}{2}t^2 \frac{\sigma^2}{2\Gamma-1}}.
$$
By L\'{e}vy's continuity theorem, this convergence of characteristic functions implies
$$
X_{n+1} \stackrel{\dd}{\longrightarrow} \nn\left(0,\frac{\sigma^2}{2\Gamma-1}\right).
$$
Next, we prove that (\ref{p-2}) holds. From (\ref{X}) and (\ref{xi}) we have
\begin{align}\label{p-3}
\left|\varphi_{n+1}(t)-\beta_{n+1}(t)\right|&=\left|\ee e^{it X_{n+1}}-\beta_n\left(B_nt\right)\left(1-\frac{t^2 \sigma^2}{2(n+1)}\right)\right|\nonumber\\
&=\left|\ee e^{it B_n X_{n}+it\frac{V_{n+1}}{\sqrt{n+1}}}-\beta_n\left(B_nt\right)\left(1-\frac{t^2 \sigma^2}{2(n+1)}\right)\right|\nonumber\\
&=\left|\ee \left[\left[e^{it B_n X_{n}}-\beta_n\left(B_nt\right)\right]\left(1-\frac{t^2 \sigma^2}{2(n+1)}\right)\right.\right.\nonumber\\
&\left. \left. \ \ \ \ \ \ +e^{it B_n X_n}\left(e^{it\frac{V_{n+1}}{\sqrt{n+1}}}-1+\frac{t^2 \sigma^2}{2(n+1)}\right)\right] \right|\nonumber\\
&\le \left|1-\frac{t^2 \sigma^2}{2(n+1)}\right| \left|\varphi_{n}\left(B_nt\right)-\beta_n\left(B_nt\right)\right|\nonumber\\
&\ \ \ \ \ +\ee \left|e^{it B_n X_n} \ee\left[e^{it\frac{V_{n+1}}{\sqrt{n+1}}}-1+\frac{t^2 \sigma^2}{2(n+1)}\Big| \FF_n\right]\right|\nonumber\\
&\le \left|1-\frac{t^2 \sigma^2}{2(n+1)}\right| \left|\varphi_{n}\left(B_nt\right)-\beta_n\left(B_nt\right)\right|\nonumber\\
&\ \ \ \ \ +\ee \left| \ee\left[e^{it\frac{V_{n+1}}{\sqrt{n+1}}}-1+\frac{t^2 \sigma^2}{2(n+1)}\Big| \FF_n\right]\right|\nonumber\\
&:=\left|1-\frac{t^2 \sigma^2}{2(n+1)}\right| \big|\varphi_{n}\left(B_nt\right)-\beta_n\left(B_nt\right)\big|+\alpha_n(t),
\end{align}
where
$$
\alpha_n(t):=\ee \left| \ee\left[e^{it\frac{V_{n+1}}{\sqrt{n+1}}}-1+\frac{t^2 \sigma^2}{2(n+1)}\Big| \FF_n\right]\right|.
$$
To show that $\left|\varphi_{n+1}(t)-\beta_{n+1}(t)\right|$ tends to zero for any $t$, it suffices to fix an arbitrary  $T>0$ and consider $|t|\le T$.
By applying Lemma \ref{lem-2} with $n=2$, and from (\ref{V}), for any $\varepsilon >0$, we obtain
\begin{align}\label{p-4}
\alpha_n(t)&=\ee \left| \ee\left[e^{it\frac{V_{n+1}}{\sqrt{n+1}}}-1+\frac{t^2 \sigma^2}{2(n+1)}\Big| \FF_n\right]+\frac{t^2 }{2(n+1)}\ee\left[V_{n+1}^2\Big| \FF_n\right]-\frac{t^2 }{2(n+1)}\ee\left[V_{n+1}^2\Big| \FF_n\right]\right|\nonumber\\
&\le \ee \left| \ee\left[e^{it\frac{V_{n+1}}{\sqrt{n+1}}}\Big| \FF_n\right]-1+\frac{t^2 }{2(n+1)}\ee\left[V_{n+1}^2\Big| \FF_n\right]\right|+\frac{t^2 }{2(n+1)} \ee \left|\sigma^2-\ee\left[ V_{n+1}^2\Big| \FF_n\right]\right|\nonumber\\
&\le \ee \min\left\{\frac{|tV_{n+1}|^2}{n+1}, \frac{|tV_{n+1}|^3}{6(n+1)^{\frac{3}{2}}}\right\}+\frac{t^2 }{2(n+1)} \ee \left|\sigma^2-\ee\left[ V_{n+1}^2\Big| \FF_n\right]\right|\nonumber\\
&\le \ee\left[\frac{|tV_{n+1}|^2}{n+1} \ii_{\left\{V_{n+1}^2\ge \varepsilon (n+1)\right\}}\right]+\ee\left[\frac{|tV_{n+1}|^3}{6(n+1)^{\frac{3}{2}}} \ii_{\left\{V_{n+1}^2< \varepsilon(n+1)\right\}}\right]\nonumber\\
&\ \ \ \ \ \  +\frac{t^2 }{2(n+1)} \ee \left|\sigma^2-\ee\left[ V_{n+1}^2\Big| \FF_n\right]\right|\nonumber\\
&\le t^2\frac{\ee\left[\left|V_{n+1}\ii_{\left\{V_{n+1}^2\ge \varepsilon (n+1)\right\}}\right|^2\right]}{n+1} +|t|^3\frac{\ee\left[|V_{n+1}|C_V\ii_{\left\{V_{n+1}^2< \varepsilon(n+1)\right\}}\right]}{6(n+1)^{\frac{3}{2}}}\nonumber\\
&\ \ \ \ \ \  +\frac{t^2 }{2(n+1)} \ee \left|\sigma^2-\ee\left[ V_{n+1}^2\Big| \FF_n\right]\right|\nonumber\\
&\le t^2\frac{\ee\left[\left|V_{n+1}\ii_{\left\{V_{n+1}^2\ge \varepsilon (n+1)\right\}}\right|^2\right]}{n+1} +|t|^3\frac{\sqrt{\varepsilon}C_V}{6(n+1)}+\frac{t^2 }{2(n+1)} \ee \left|\sigma^2-\ee\left[ V_{n+1}^2\Big| \FF_n\right]\right|\nonumber\\
&:=|t| \frac{g_n(t)}{n+1},
\end{align}
where
$$
g_n(t):=|t|\ee\left[\left|V_{n+1}\ii_{\left\{V_{n+1}^2\ge \varepsilon (n+1)\right\}}\right|^2\right]+t^2 \frac{\sqrt{\varepsilon}C_V}{6}+|t| \ee \left|\sigma^2-\ee\left[ V_{n+1}^2\Big| \FF_n\right]\right|.
$$
From (\ref{V}), (\ref{V-1}) and dominated convergence theorem, for any fixed $t$, we have
\beq\label{g}
\lim_{n\rightarrow\infty} g_n(t)=t^2 \frac{\sqrt{\varepsilon}C_V}{6}.
\deq

Let
$$
k_n(t):=\left|\varphi_{n}(t)-\beta_{n}(t)\right|.
$$
From (\ref{p-3}) and (\ref{p-4}), we can conclude that
\beq\label{k}
k_{n+1}(t)\le \left|1-\frac{t^2 \sigma^2}{2(n+1)}\right| k_n(B_nt)+ |t| \frac{g_n(t)}{n+1}.
\deq
Below, we will primarily prove for any $|t|\le T$, $\displaystyle\lim_{n\rightarrow\infty} k_n(t)=0$. We can choose a sufficiently large $N_0$ such that for all $n\ge N_0$, $ B_n=1-\frac{\Gamma-\frac{1}{2}}{n+1}\in (0,1)$ and note that $g_n(t)$ is increasing in $|t|$, we can get that
\beq\label{k1}
k_{n+1}(t)\le  k_n(B_nt)+ |t| \frac{g_n(T)}{n+1}.
\deq
Note that
$$
k_1(t)=\left|\varphi_{1}(t)-\beta_{1}(t)\right|=\left|\ee e^{it X_1}-1\right|,
$$
by applying Lemma \ref{lem-2} with $n=0$, we can get that
$$
k_1(t)\le \ee \min\{2, |t X_1|\}\le |t| \ee|X_1|.
$$
Hence we can get that
$$
k_1(t)=O(|t|) \ \ \text{as} \ \ t\rightarrow0.
$$
From (\ref{k1}) we can get that for all $n\ge N_0$,
$$
k_n(t)=O(|t|)\ \ \ \text{as}\ \ \ t\rightarrow0.
$$
 For the initial index $N_0$, define
$$
\delta_{N_0}(T):=\sup_{-T\le t\le T} \frac{k_{N_0}(t)}{|t|},
$$
then we know that $\delta_{N_0}(T)$ is finite. For any $|t|\le T$, note that $B_{N_0}\in (0,1)$, we have
\beq\label{k0}
k_{N_0}(t)\le |t|\delta_{N_0}(T)\ \ \text{and}\ \ k_{N_0}\left(B_{N_0}t\right)\le B_{N_0}|t|\delta_{N_0}(T).
\deq
For all $j\ge N_0$, define $\delta_{j+1}(T)$ recursively by
$$
\delta_{j+1}(T):=B_j \delta_j(T)+\frac{g_j(T)}{j+1}=\left(1-\frac{\Gamma-\frac{1}{2}}{j+1}\right)\delta_j(T)+\frac{g_j(T)}{j+1}.
$$
From (\ref{g}) and Lemma \ref{cor-1}, we can get that
\beq\label{d-1}
\lim_{j\rightarrow\infty} \delta_{j+1}(T)=T^2 \frac{\sqrt{\varepsilon}C_V}{6\Gamma-3}.
\deq
Below we employ the method of inductive hypothesis. If we assume that (\ref{k0}) holds for $j$ in place of $N_0$, i.e. we assume that
$$
k_j(t)\le |t| \delta_{j} (T) \ \ \text{and}\ \ k_{j}\left(B_{j}t\right)\le B_{j}|t|\delta_{j}(T)
$$
holds. Together with (\ref{k1}), we have
$$
k_{j+1}(t)\le  k_j\left(B_j t\right)+|t|\frac{g_j(T)}{j+1}\le|t|B_j \delta_j(T)+|t|\frac{g_j(T)}{j+1}=|t|\delta_{j+1}(T).
$$
Therefore, we can conclude that the assumption holds for $j+1$. As a consequence, since $B_{j+1}\in (0,1)$, we can also get that
$$
k_{j+1}\left(B_{j+1}t\right)\le B_{j+1}|t|\delta_{j+1}(T).
$$
Hence we can get that for all $|t|\le T$ and $j\ge N_0$,
$$
k_{j}(t)\le |t|\delta_{j}(T).
$$
Together with (\ref{d-1}) and considering that $\varepsilon$ can take arbitrarily small values, we obtain that
$$
\lim_{j\rightarrow\infty} k_{j}(t)=0.
$$
Hence we know that
$$
\varphi_n(t)\rightarrow \beta_{n}(t),\ \ \text{as}\ \ n\rightarrow\infty,\ \ \text{for all}\ \ t.
$$

{\bf Step 2. Detailed Proof:} We now drop the two simplifying assumptions made in Step 1 and return to the original recursion (\ref{X-1}).
Let
\beq\label{D}
D_{n+1}:=\left(1-\frac{\Gamma-\frac{1}{2}}{n+1}\right) D_n+\frac{V_{n+1}-\ee \left[V_{n+1}\big|\FF_n\right]}{\sqrt{n+1}},
\deq
then $D_n$ satisfies the assumptions of Step 1 and from Step 1 we have
$$
D_{n+1} \stackrel{\dd}{\longrightarrow} \nn\left(0,\frac{\sigma^2}{2\Gamma-1}\right).
$$
To prove that (\ref{p-1}) still holds, we construct the process $\left(D_n\right)_{n\ge 1}$ that satisfies the  Step 1  conditions and show that the difference $\Delta_n:= X_n-D_n$ converges to zero in probability.
From (\ref{X-1}) we have
$$
\aligned
\Delta_{n+1}:&=X_{n+1}-D_{n+1}\\
&=\left(1-\frac{\Gamma_{n+1}-\frac{1}{2}+O\left(\frac{1}{n+1}\right)}{n+1}\right)X_n-\left(1-\frac{\Gamma-\frac{1}{2}}{n+1}\right) D_n+\frac{\ee \left[V_{n+1}\big|\FF_n\right]}{\sqrt{n+1}}\\
&=\left(1-\frac{\Gamma_{n+1}-\frac{1}{2}+O\left(\frac{1}{n+1}\right)}{n+1}\right)\Delta_n+\frac{\Gamma-\Gamma_{n+1}+O\left(\frac{1}{n+1}\right)}{n+1}D_n+\frac{\ee \left[V_{n+1}\big|\FF_n\right]}{\sqrt{n+1}}.
\endaligned
$$
If we can prove $\Delta_{n+1}\stackrel{P}{\longrightarrow} 0$, then by Lemma \ref{lem-3}, we can conclude that
$$
X_{n+1} \stackrel{\dd}{\longrightarrow} \nn\left(0,\frac{\sigma^2}{2\Gamma-1}\right).
$$
holds. Below, we will focus on proving $\Delta_{n+1}\stackrel{L_1}{\longrightarrow} 0$, which will imply $\Delta_{n+1}\stackrel{P}{\longrightarrow} 0$.
From (\ref{D}) we know that
\beq\label{D-1}
\aligned
D_{n+1}^2&=\left(1-\frac{\Gamma-\frac{1}{2}}{n+1}\right)^2 D_n^2+\frac{\left(V_{n+1}-\ee \left[V_{n+1}\big|\FF_n\right]\right)^2}{n+1}\\
&\ \ \ \ \ \ \ \ \ \ \ \ \ +2\left(1-\frac{\Gamma-\frac{1}{2}}{n+1}\right) D_n\frac{V_{n+1}-\ee \left[V_{n+1}\big|\FF_n\right]}{\sqrt{n+1}}.
\endaligned
\deq
Note that
$$
\aligned
\ee\left[D_n\left(V_{n+1}-\ee \left[V_{n+1}\big|\FF_n\right]\right)\right]&=\ee\left[\ee\left[D_n\left(V_{n+1}-\ee \left[V_{n+1}\big|\FF_n\right]\right)\Big|\FF_n\right]\right]\\
&=\ee\left[D_n\left(\ee\left[V_{n+1}\big|\FF_n\right]-\ee \left[V_{n+1}\big|\FF_n\right]\right)\right]=0.
\endaligned
$$
Therefore, taking the expected value of both sides of (\ref{D-1}) yields
$$
\aligned
\ee D_{n+1}^2&=\left(1-\frac{\Gamma-\frac{1}{2}}{n+1}\right)^2 \ee D_n^2+\frac{\ee \left[\left(V_{n+1}-\ee \left[V_{n+1}\big|\FF_n\right]\right)^2\right]}{n+1}\\
&=\left(1-\frac{2 \Gamma-1-\frac{\left(\Gamma-\frac{1}{2}\right)^2}{n+1}}{n+1}\right) \ee D_n^2+\frac{\ee \left[ \left(V_{n+1}-\ee \left[V_{n+1}\big|\FF_n\right]\right)^2\right]}{n+1}.
\endaligned
$$
From (\ref{V}) we know that $|V_{n+1}|\le \sqrt{C_V}$. Construct a sequence $\{y_n\}_{n\ge 1}$ such that $y_1\ge \ee D_{1}^2$ and $y_n$ satisfies
$$
y_{n+1}=\left(1-\frac{2 \Gamma-1-\frac{\left(\Gamma-\frac{1}{2}\right)^2}{n+1}}{n+1}\right) y_n+\frac{4 C_V}{n+1}.
$$
Below we employ the method of inductive hypothesis. Suppose that $y_k\ge \ee D_k^2$. Then we have
$$
\aligned
\ee D_{k+1}^2&=\left(1-\frac{2 \Gamma-1-\frac{\left(\Gamma-\frac{1}{2}\right)^2}{k+1}}{k+1}\right) \ee D_k^2+\frac{\ee \left(V_{k+1}-\ee \left[V_{k+1}\big|\FF_k\right]\right)^2}{k+1}\\
&\le \left(1-\frac{2 \Gamma-1-\frac{\left(\Gamma-\frac{1}{2}\right)^2}{k+1}}{k+1}\right) y_k +\frac{4 C_V}{k+1}=y_{k+1}.
\endaligned
$$
Hence we can get that for all $n\ge 1, y_n\ge \ee D_n^2$. From Lemma \ref{cor-1} we can get that
$$
\lim_{n\rightarrow\infty} y_n=\frac{4 C_V}{2 \Gamma-1},
$$
then we have
$$
\limsup_{n\rightarrow\infty} \ee D_{n+1}^2\le\limsup_{n\rightarrow\infty} y_{n+1} = \lim_{n\rightarrow\infty} y_{n+1}=\frac{4 C_V}{2 \Gamma-1}.
$$
From Cauchy-Schwarz inequality we know
\beq\label{D-3}
\limsup_{n\rightarrow\infty} \ee \left|D_{n+1}\right|\le \sqrt{\frac{4 C_V}{2 \Gamma-1}}.
\deq
From (\ref{V-2}) we know that
$$
\left|\ee \left[V_{n+1}\big|\FF_n\right]\right|=\left|\ee\left[\frac{(n+1)\Delta M_{n+1}}{T_{n+1}}\Big|\FF_n\right]\right|\le \frac{2K}{n},
$$
hence we have
$$
\frac{\left|\ee \left[V_{n+1}\big|\FF_n\right]\right|}{\sqrt{n+1}}= o\left(\frac{1}{n+1}\right)\ \ \ \ a.s.
$$
From (\ref{T4}) we know that there exist a positive constant $N_1$ such that for all $n\ge N_1$,
$$
\left|\Gamma_{n+1}-\frac{1}{2}+O\left(\frac{1}{n+1}\right)-\left(\Gamma-\frac{1}{2}\right)\right|\le \varepsilon
$$
and
$$
\left|\Gamma-\Gamma_{n+1}+O\left(\frac{1}{n+1}\right)\right|\le \varepsilon,
$$
here $\varepsilon$ is a positive constant small enough. Then for all $n\ge N_1$, we can get that
\beq\label{D-2}
\left|\Delta_{n+1}\right|\le \left(1-\frac{\Gamma-\frac{1}{2}-\varepsilon}{n+1}\right)\left|\Delta_{n}\right|+\frac{\varepsilon}{n+1}\left|D_n\right|+o\left(\frac{1}{n+1}\right).
\deq
Taking the expected value of both sides of (\ref{D-2}) yields and from (\ref{D-3}) we have
$$
\aligned
\ee \left|\Delta_{n+1}\right|&\le \left(1-\frac{\Gamma-\frac{1}{2}-\varepsilon}{n+1}\right)\ee \left|\Delta_{n}\right|+\frac{\varepsilon}{n+1}\ee \left|D_n\right|+o\left(\frac{1}{n+1}\right)\\
&\le \left(1-\frac{\Gamma-\frac{1}{2}-\varepsilon}{n+1}\right)\ee \left|\Delta_{n}\right|+\frac{\varepsilon \sqrt{\frac{4 C_V}{2 \Gamma-1}}+o(1)}{n+1}.
\endaligned
$$
Construct a sequence $\{x_n\}_{n\ge 1}$ such that $x_1\ge \ee \left|\Delta_{1}\right|$ and $x_n$ satisfies
$$
x_{n+1}=\left(1-\frac{\Gamma-\frac{1}{2}-\varepsilon}{n+1}\right) x_n+\frac{\varepsilon \sqrt{\frac{4 C_V}{2 \Gamma-1}}+o(1)}{n+1}.
$$
According to the inductive hypothesis method, we can conclude that for all $n\ge 1$, $\ee \left|\Delta_{n}\right|\le x_n$ holds true. From Lemma \ref{cor-1} we can get that
$$
\lim_{n\rightarrow\infty} x_n=\frac{2 \varepsilon \sqrt{C_V}}{\sqrt{2 \Gamma-1}\left(\Gamma-\frac{1}{2}-\varepsilon\right)}.
$$
Hence we have
$$
\limsup_{n\rightarrow\infty} \ee\left|\Delta_{n}\right|\le \lim_{n\rightarrow\infty} x_n=\frac{2 \varepsilon \sqrt{C_V}}{\sqrt{2 \Gamma-1}\left(\Gamma-\frac{1}{2}-\varepsilon\right)}.
$$
Note that $\varepsilon$ small enough, taking $\varepsilon\rightarrow 0$, we can get that
$$
\ee\left|\Delta_{n}\right|\rightarrow 0,
$$
which makes $\Delta_{n}\stackrel{P}{\longrightarrow} 0$.
\end{proof}

\end{document}